\documentclass{amsart}
\usepackage{amsfonts, amsmath, amssymb, amsthm}

\newtheorem{theorem}{Theorem}[section]

\newtheorem{lemma}[theorem]{Lemma}

\newtheorem{definition}[theorem]{Definition}

\newcommand{\beq} {\begin{equation}}
\newcommand{\eeq} {\end{equation}}

\begin{document}
\title[Non-homogeneity of Remainders]{On Cardinality of Non Isomorphic Intermediate Rings of $C(X)$}
\author{Bedanta Bose}
\address{Swami Niswambalananda Girls' College\\ 115, B.P.M.B. Sarani, Bhadrakali\\ Hooghly - 712232\\ India}
\email{ana\_bedanta@yahoo.com; anabedanta@gmail.com}

\begin{abstract}
Let $\sum (X)$ be the collection of subalgebras of $C(X)$ containing $C^{*}(X)$, where $X$ is a Tychonoff space. For any $A(X)\in \sum(X)$ there is associated a subset $\upsilon_{A}(X)$ of $\beta X$ which is an $A$-analogue of the Hewitt real compactification $\upsilon X$ of $X$. For any $A(X)\in \sum(X)$, let $[A(X)]$ be the class of all $B(X)\in \sum(X)$ such that $\upsilon_{A}(X)=\upsilon_{B}(X)$. We have shown that for first countable  non compact real compact space $X$, $[A(X)]$ contains at least $2^{c}$ many different subalgebras no two of which are isomorphic.
\end{abstract}

\maketitle
\section{Introduction}
For a completely regular Hausdroff space $X$, let $C(X)$ and $C^{*}(X)$ denote the rings of all real valued and bounded real valued continuous functions on $X$ respectively. We denote  `intermediate ring' by $A(X)$ which mean that $C^{*}(X) \subseteq A(X) \subseteq C(X)$. Let $\sum (X)$ denote the class of all intermediate rings.

As we dig into the literature the study of intermediate rings started when D.Plank in 1969 \cite{plank} established a result that the structure space of any intermediate ring is homeomorphic to the Stone-Cech compactification of $X$, i.e., $\beta X$. The proof is quite different and complicated in compare with the conventional method of showing the same result for $C(X)$ only. In 1987 Redlin and Watson \cite{ls} proved the same result but using a new technique. They have created a new operator, called $\mathcal{Z}_{A}$ for an intermediate ring $A(X)$, which identify with one to one basis the maximal ideals of $A(X)$ with $z$-ultrafilters of $X$. They also established Gelfand-Kolmogorov like characterization for maximal ideals using the points of $\beta X$. Following their notation we denote maximal ideals of $A(X)$ by $M_{A}^{p}$ for each $p\in \beta X$.

%In rings of continuous functions study of the rings $C(X)$ and $C^{*}(X)$ have been done extensively. The main central result of this study is the fact that the structure spaces of these rings are homeomorphic to the Stone-Cech compactification of $X$, i.e., $\beta X$. Depending on this fact several algebraic as well as topological properties of $C(X)$ and $C^{*}(X)$ have been investigated till now. Study of these properties in case of rings lying between $C^{*}(X)$ and $C(X)$, i.e., `intermediate rings' was first initiated by D.Plank in 1969 \cite{plank}. One among all the major results he proved is that the structure space of intermediate rings are homeomorphic to the Stone-Cech compactification $\beta X$- proof of which was quite different with the usual way of proving this result for $C(X)$ by establishing relation between maximal ideals of $C(X)$ and $z$-ultrafilters of $X$ by the operator $Z$ which is defined by taking each function to its zero set. In 1987 Redlin and Watson \cite{ls} proved the same by inventing a new operator $\mathcal{Z}_{A}$ for an intermediate ring $A(X)$ that keep the relation between maximal ideals of $A(X)$ and $z$-ultrafilters of $X$ with a little modification and generalizes the operator $Z$. Research along this line by Redlin, Watson, Buyun, Pannman, Sack, etc developed machineries to deal with problems. In this context mention may be made of \cite{hs}, \cite{ls}, \cite{lsw}, \cite{JSack}- few of their works.

For $A(X)\in \sum(X)$ we associate a subspace $\upsilon_{A}(X)\subseteq \beta X$ which is an $A$-analogue of the Hewitt real compactification $\upsilon X$ of $X$, i.e., the collection of all points  $p\in \beta X$ for which $A(X)/M_{A}^{p}$ is isomorphic to $\mathbb{R}$ with the usual subspace topology of $\beta X$. We define an equivalence relation $A(X) \sim B(X)$  if and only if $\upsilon_{A}(X)$ is homeomorphic to $\upsilon_{B}(X)$ where $A(X)$, $B(X)\in \sum(X)$. We denote the equivalence class for $A(X)\in\sum(X)$ by $[A(X)]$. The cardinality of a particular class $[A(X)]$ and algebraic inter relation between any two rings from a same class is merely a question. Redlin and Watson gave an example in \cite{ls} of an intermediate class where at least two rings are non isomorphic. We incorporate this example in verbatim as follows: let $H(\mathbb{N})$ be the the algebra of sequences which occur as the coefficients of the Taylor series representation of functions holomorphic on
the open unit disc. Then $\mathbb{N}$ is both $H$-compact (see \cite{Br}) and $C$-compact, but $H(\mathbb{N})$ is obviously not isomorphic to $C(\mathbb{N})$. This particular observation is the principal motivation in this article to search for cardinality of non isomorphic rings in a particular class $[A(X)]$.

%Existence of non isomorphic rings in a particular class $[A(X)]$ is shown by Redlin and Watson in \cite{ls}. This particular observation is the principal motivation to search for number of non isomorphic rings in a particular class $[A(X)]$.

Research along this line by Redlin, Watson, Buyun, Pannman, Sack, etc developed machineries to which one can deal with problems. In this context mention may be made of \cite{hs}, \cite{ls}, \cite{lsw}, \cite{JSack} which are few of their works.
 
In the second section of this article we have established a result [Theorem \ref{proper copy}] that tells that a real closed $\eta_{\alpha}$ field of power $\aleph_{\alpha}$ contains a proper copy of itself, which provides a way to identify subring $B_{p}^{\mathfrak{F}}(X)\in \sum(X)$ of the class $[A(X)]$ where $\mathfrak{F}$ is a real closed $\eta_{1}$-field with transcendence base at least $c$ and $p\in \beta X\setminus \upsilon_{A^{\nu}}(X)$. As this identification $B_{p}^{\mathfrak{F}}(X)$ of subrings depends upon the point $p\in \beta X\setminus \upsilon_{A^{\nu}}(X)$ it can be conclude that there are plenty of subrings, in fact, at least $2^c$ many in the class $[A(X)]$ [Theorem \ref{distinct rings}]. In section 3 we show that the existence of isomorphic subrings in the class $[A(X)]$ is equivalent to the existence of homeomorphism from $\beta X$ onto itself which therefore relates our main focus to homogeneity of $\beta X$ [Theorem \ref{characterization},Theorem \ref{characterization2}]. Finally we are able to show that there exists at least $2^c$ many non-isomorphic different subrings in the class $[A(X)]$ if $X$ is first countable, locally compact, non-pseudocompact and realcompact which contains a $C$-embedded copy of $\mathbb{N}$.

Some results and notations of intermediate rings have been used in this paper. For clarity purpose one can find it in \cite{hs}, \cite{ls}, \cite{lsw}, \cite{JSack}.

\section{Some Basic Results and Discussion}

Recall $\upsilon_A X=\{p\in \beta X~|~ A(X)/ M_{A}^{p}~ \mbox{ isomorphic to}~ \mathbb{R}\}$. Let for each $A(X)\in \sum(X)$,  $A^{\nu}(X)=\{f\in C(X)~|~f^{*}(\upsilon_{A}X)\subseteq \mathbb{R}\}$ where $f^*$ is the unique stone-extension of $f$ from $\beta X$ to the one-point compactification of $\mathbb{R}$, i.e., $\mathbb{R}^{*}$. It is obvious from Theorem 3.5 of \cite{hs} that $A(X)\subseteq A^{\nu}(X)$ and also $C^*(X)\subseteq A^{\nu}(X)$ which therefore imply  that $A^{\nu}(X)$ is an intermediate subalgebra of $C(X)$. Also it is clear from the definition of $A^{\nu}(X)$ that  $\upsilon_{A} X= \upsilon_{A^{\nu}} X$ and therefore $A(X)\subseteq A^{\nu}(X)$ for all $A(X)\in [A(X)]$. In a nutshell $A^{\nu}(X)$ is the largest element under set inclusion in the class $[A(X)]$ and also $A^{\nu}(X)\cong C(\upsilon_{A}(X))$ \cite{lsw}, i.e., $A^{\nu}(X)$ is a $C$-ring and hence we can directly use Theorem 13.2, Theorem 13.4 of \cite{lm} and conclude the following result.

\begin{theorem}\label{real closed}
Every hyper-real field of the form $A^{\nu}(X)/M_{A^{\nu}}^{p}$ for $p\in \beta X\setminus \upsilon_{A^{\nu}}(X)$ is a real closed $\eta_{1}$-field with transcendence base at least $c$.
\end{theorem}
%\begin{proof}
%Let $f_{1},f_{2},...,f_{n}\in A^{\nu}(X)$. Consider the function $\phi:\upsilon_{A}(X)\longrightarrow \mathbb{R}^{n}$ defined by $\phi(x)=(f^{*}_{1}(x),f^{*}_{2}(x),...,f^{*}_{n}(x))$ for all $x\in \upsilon_{A}(X)$. Then $\phi$ will be a continuous function (easy to check). Let $$p_{x}(\lambda)=\lambda^{n}+f_{1}(x)\lambda^{n-1}+...+f_{n}(x)$$. Let $\omega_{i}(x)$ denotes the ith element after ordering the real part of all the roots of $p_{x}(\lambda)$. Let $\omega_{i}\circ\phi=g_{i}$. Then using the similar arguements as in the proof to show $C(X)/M^{p}$ is real closed in GJ we can conclude that $p(g_{i})\equiv 0$ modulo $M_{A^{\nu}}^{p}$ for some $i\in \mathbb{N}$.   
%\end{proof}

We know that real closed $\eta_{\alpha}$-fields of cardinality $\aleph_{\alpha}$ are isomorphic \cite{PLM}. This fact takes a crucial role to prove the following result.

\begin{theorem} \label{proper copy} Let $\alpha>0$ be any ordinal and let $\mathfrak{F}$ be a real closed $\eta_{\alpha}$-field of power $\aleph_{\alpha}$. Then $\mathfrak{F}$ contains a proper copy of itself.
\end{theorem}
\begin{proof} By Lemma 13.11 \cite{lm} $\mathfrak{F}$ has a dense transcendense base over $\mathbb{Q}$ and let it be $A$. Then $\mathfrak{F}=\mathfrak{R}(\mathbb{Q}(A))$. Let $a\in A$ and $A'=A-\{a\}$. Then $A'$ is also dense in $\mathfrak{F}$ and $|A'|=|A|$. Now let $\mathfrak{F}'=\mathfrak{R}(\mathbb{Q}(A'))$. Obviously, $\mathfrak{F}'\subset \mathfrak{F}$. This inclusion is proper, otherwise $A'$ will be the maximal set of independent transcendentals and hence a base which contradicts our assumption that $A$ is a base. Again, $A'\subseteq \mathfrak{F}'$ imply that $\mathfrak{F}'$ is dense in $\mathfrak{F}$. Since every dense subset of a $\eta_{\alpha}$-set is a $\eta_{\alpha}$-set [lemma 1.3, \cite{PLM}], therefore, $\mathfrak{F}'$ is a $\eta_{\alpha}$-set of power $\aleph_{\alpha}$. Hence $\mathfrak{F}'$ is a real closed $\eta_{\alpha}$-field. If we assume continuum hypothesis then any two real closed field that are $\eta_{\alpha}$-set of power $\aleph_{\alpha}$ are isomorphic and hence $\mathfrak{F}$ and $\mathfrak{F}'$ are isomorphic. Hence $\mathfrak{F}$ contains a proper copy of itself.
\end{proof}

Theorem \ref{real closed} and Theorem \ref{proper copy} together ensure the fact that $A^{\nu}(X)/ M_{A^{\nu}}^{p}$ contains a proper copy of itself. Let us take $\mathfrak{F}$ be the proper copy of $A^{\nu}(X)/ M_{A^{\nu}}^{p}$ into itself and $\theta$ be the canonical map from $A^{\nu}(X)$ to the hyper-real field $A^{\nu}(X)/ M_{A^{\nu}}^{p}$, i.e., $\theta(f)=M_{A^{\nu}}^{p}(f)$ for $f\in A^{\nu}(X)$. Then $\theta^{-1}(\mathfrak{F})$ is a proper subring of $A^{\nu}(X)$ and we denote this subring by $B_{p}^{\mathfrak{F}}(X)$. Since $\mathfrak{F}$ contains a copy of $\mathbb{R}$ it follows that $C^{*}(X)\subseteq B_{p}^{\mathfrak{F}}(X)$ and hence $B_{p}^{\mathfrak{F}}(X)\in \sum(X)$. Also note that $\theta^{-1}(0)=M_{A}^{p}\subseteq B_{p}^{\mathfrak{F}}(X)$.

\begin{theorem} \label{Biswajit intermediate} Let $A(X)\in \sum(X)$ and $p,q\in \beta X$ such that $p\in \beta X-\upsilon_{A}X$ then there exists $f\in A(X)$ such that $f\in M_{A}^{q}$ and $f^{*}(p)=\infty$. 
\end{theorem}

\begin{proof}
Since $p\in \beta X-\upsilon_{A}X$, there exist $g\in A(X)$ such that $|M_{A}^{p}(g)|$ is infinitely small and $M_{A}^{p}(g)\neq 0$ in the field $A(X)/M_{A}^{p}$ and hence there exist $\xi \in A(X)$ such that $M_{A}^{p}(g)M_{A}^{p}(\xi)=1$. Since $p\neq q$, there exist an open set $V$ in $\beta X$ such that $q\in V\subseteq \mbox{cl}_{\beta X}V\subseteq \beta X-\{p\}$ and by complete regularity there exists $h\in C^{*}(X)$ such that $h^{\beta}(\mbox{cl}_{\beta X}V)=0$ and $h^{\beta}(p)=1$. Let $f=h\xi\in A(X)$. Then $|M_{A}^{p}(fg-1)|=|M_{A}^{p}(h-1)|$. Again $h^{\beta}(p)=1$ shows that $|M_{A}^{p}(h-1)|$ is either infinitely small or zero, i.e., $(gf)^{\beta}(p)=1$. Since $g^{\beta}(p)=0$, we can conclude that $f^{\beta}(p)=\infty$. Again $h^{\beta}(\mbox{cl}_{\beta X}V)=0$ and $q\in V$ therefore $h\in M_{A}^{q}$. 
\end{proof}

It is quite clear from  Theorem \ref{Biswajit intermediate} that if $f\in M_{A^{\nu}}^{q}$ then $|M_{A^{\nu}}^{q}(f)|=0$ and hence $f\in B_{q}^{\mathfrak{F}}(X)$. Again $f^{\beta}(p)=\infty$, i.e., $p\notin \upsilon_{B_{q}^{\mathfrak{F}}}(X)$ for all $p\in \beta X-\upsilon_{A^{\nu}}X$ and consequently $\upsilon_{B_{q}^{\mathfrak{F}}}(X)\subseteq \upsilon_{A^{\nu}}(X)$. Also $\upsilon_{A^{\nu}}(X)\subseteq\upsilon_{B_{q}^{\mathfrak{F}}}(X)$ which follows from the fact that $B_{q}^{\mathfrak{F}}(X)\subseteq A^{\nu}(X)$ and therefore we can conclude that $B_{q}^{\mathfrak{F}}(X)\in [A(X)]$.

Now one can conclude that for each point $p\in \beta X\setminus \upsilon_{A}(X)$ there is a subring $B_{p}^{\mathfrak{F}}(X)$ which belongs to the class $[A(X)]$. Again for locally compact, non-compact but realcompact space $X$, $\beta X\setminus X$ contains at least $2^{c}$ many elements [\cite{lm},\S Corollary 9.12] that combining with the previous fact produce the following theorem.

\begin{theorem}\label{distinct rings}
For locally compact, non-compact but realcompact space $X$, each class $[A(X)]$ contains at least $2^{c}$ many distinct subrings.
\end{theorem}

\section{Non isomorphic subalgebras of the class [$A(X)$]}

%The topological association of the subrings $B_{p}^{\mathfrak{F}}(X)$ in the class $[A(X)]$ with the points  $p\in \beta X\setminus \upsilon_{A}(X)$ shows the existence of plenty of subrings in the class $[A(X)]$ [Theorem \ref{distinct rings}]. 

This section is focused on prmary goal of this article, i.e., on the non isomorphic characteristic among the class of subrings $[A(X)]$. To do this the topological association of points $p\in \beta X\setminus \upsilon_{A}(X)$ with the subrings $B_{p}^{\mathfrak{F}}(X)$ in the class $[A(X)]$ take  the main key role. In fact it reveals the intimate relationship between the two important structures, viz.,non-isomorphism of two rings in $[A(X)]$ and non-homogeneity of $\beta X\setminus X$.    

%A this point one can investigate the following natural question due to Theorem \ref{distinct rings}: are all distinct elements in the class $[A(X)]$ of the type $B_{p}^{\mathfrak{F}}(X)$ algebraically same (upto isomorphism)? To answer this we show in the next theorem that there is an intimate relation between the isomorphism of the above mentioned rings and homogeneity of the $\beta X$.
%\begin{remark} From the proof of Theorem \ref{proper copy} it is clear that for every proper subset $B$ of a dense transcendental base $A$ of the hyper real real closed ring $\mathfrak{F}$, which is also dense in $\mathfrak{F}$ gives a proper copy of $\mathfrak{F}$ into itself, say $\mathfrak{F}'$, and from that copy $\mathfrak{F}'$ we can construct an intermediate ring $B_{p}^{\mathfrak{F}'}(X)$ of the class $[A(X)]$ corresponding to any point $p\in \beta X - \upsilon_{A}X$. Therefore there is an one to one relation from the subsets of $A$ which are dense in $\mathfrak{F}$ and elements of a class $[A(X)]$ which are subrings of $A(X)$ and contains $M_{A}^{p}$ as a maximal ideals. (??)(verify one to one prop).
%\end{remark}
\begin{theorem} \label{characterization} Let $p,q\in \beta X - \upsilon_{A}X $. Then the two rings $B_{p}^{\mathfrak{F}}(X)$ and $B_{q}^{\mathfrak{F}'}(X)$ are isomorphic if there is a homeomorphism from $\beta X$ to $\beta X$ which takes $p$ to $q$ and induces an isomorphism from ${\mathfrak{F}}$ onto ${\mathfrak{F}'}$.
\end{theorem}
\begin{proof} Let $B_{p}^{\mathfrak{F}}(X)$ and $B_{q}^{\mathfrak{F}'}(X)$ are isomorphic for $p,q\in \beta X - \upsilon_{A}X$ and $\phi$ be the isomorphism. Let $\theta_{p}$ be the corresponding canonical map, as mentioned earlier, from $A^{\nu}(X)$ to $A^{\nu}(X)/M_{A^{\nu}}^{p}$ for the ring $B_{p}(X)$ and similarly $\theta_{q}$ for the ring $B_{q}(X)$. Let $\mathfrak{F}$ and $\mathfrak{F}'$ are the isomorphic copies in $A^{\nu}(X)/M_{A^{\nu}}^{p}$ and $A^{\nu}(X)/M_{A^{\nu}}^{q}$ respectively such that $B_{p}^{\mathfrak{F}}(X)=\theta_{p}^{-1}(\mathfrak{F})$ and $B_{q}(X)=\theta_{q}^{-1}(\mathfrak{F}')$. Then $\theta_{q}\phi\theta_{p}^{-1}=\phi$ gives an isomomorphism from $\mathfrak{F}$ to $\mathfrak{F}'$ and hence it takes zero to zero and as a consequence we have $\phi(M_{A^{\nu}}^{p})=M_{A^{\nu}}^{q}$.

From the result of D.Rudd [\cite {Rudd}, \S Corollary 3.6] it follows that if $M_{A^{\nu}}^{p}$ is a maximal ideal of $A^{\nu}(X)$ then maximal ideals of $M_{A^{\nu}}^{p}$ are precisely $M_{A^{\nu}}^{p}\cap M$ where $M$ is a maximal ideal of $A^{\nu}(X)$ and $M$ does not contain $M_{A^{\nu}}^{p}$. Let $\beta M_{A^{\nu}}^{p}=\{M_{A^{\nu}}^{p}\cap M: M\not\supset M_{A^{\nu}}^{p}~ \mbox{and}~ M~ \mbox{is a maximal in}~ A^{\nu}(X)\}$. Then $\beta M_{A^{\nu}}^{p}$ is the collection of all maximal ideals of $M_{A^{\nu}}^{p}$ and it admits hull kernel topology. Now the mapping $\tau:\beta M_{A^{\nu}}^{p}\rightarrow \mathcal{M}$, defined by $\tau(M_{A^{\nu}}^{p}\cap M)=M$ is a homeomorphism into $\mathcal{M}$, where $\mathcal{M}$ is the space of maximal ideals of $A^{\nu}(X)$. Since $\mathcal{M}$ is homeomorphic to $\beta X$ [\cite{ls},\S Theorem 4] and $\tau(\beta M_{A^{\nu}}^{p})$ is open in $\mathcal{M}$ [\cite{Rudd}, \S 2.3] therefore $\beta M_{A^{\nu}}^{p}$ is locally compact and Hausdorff. The one-point compactification of $\beta M_{A^{\nu}}^{p}$ is evidently homeomorphic to $\beta X\setminus \{p\}$. Again the structure spaces of $M_{A}^{p}$ and $M_{A}^{q}$ are homeomorphic and hence the isomorphism $\phi$ gives a homeomorphism from $\beta X-\{p\}$ to $\beta X-\{q\}$. Since $\beta X-\{p\}$ and $\beta X-\{q\}$ are locally compact (\cite{Rudd}, \S Remark 3.9), this homeomorphism can be extended to a homeomorphism from $\beta X$ to $\beta X$ which takes $p$ to $q$.
\end{proof}

\begin{lemma}\label{induced iso}
Let $X$ be a first countable tychonoff space and suppose $\bar{\phi}: \beta X\longmapsto \beta X$ be a homeomorphism then $\bar{\phi}$ induces an isomorphism on $A^{\nu}$ onto itself.
\end{lemma}

\begin{proof}
First countability of $X$ imply that the space $\beta X$ is first countable at each point of $X$ and therefore each point of $X$ is a $G_{\delta}$- point of $\beta X$. On the other hand, no point of $\beta X-X$ can be a $G_{\delta}$-point of $\beta X$ [\cite{lm}, 9.6]. As a result the homeomorphism $\bar{\phi}$ exchanges the point of $X$, i.e., $\bar{\phi}(X)=X$. Let us denote $\bar{\phi}|_{X}=\phi$. Then $\phi$ is a homeomorphism from $X$ onto itself. The above homeomorphism induces a map, say, $\Psi :C(X) \longmapsto C(X)$, defined by $\Psi(f)=f\circ \phi^{-1}$. It is evident that $\Psi$ is a homomorphism. Again $\phi(X)=X$ imply that $Ker \Psi=\{0\}$. Since for every $f\in C(X)$, $f\circ \phi \in A^{\nu}$ and $\Psi(f\circ \phi)=(f\circ\phi)\circ\phi^{-1}=f$, $\Psi$ is onto. Hence $\Psi$ is an automorphism from $C(X)$ onto itself. To prove the result it is sufficient to show that $\Psi(A^{\nu})=A^{\nu}$. First we observe that due to the isomorphism of $\Psi$ both $A^{\nu}$ and $\Psi(A^{\nu})$ have the same real maximal ideal space and hence $\Psi(A^{\nu})\in [A(X)]$. Again $\Psi$ preserve the order (set inclusion) among the subrings of $C(X)$ which evidently show that $\Psi(A^{\nu})$ is the largest among all the subrings in the class $[A(X)]$ and hence $\Psi(A^{\nu})=A^{\nu}$.
\end{proof}

Some important properties and notations of intermediate rings are used in the following lemma which one can find in \cite{hs}.

\begin{lemma}
Let $X$ be a first countable tychonoff space and $\bar{\phi}:\beta X\longrightarrow \beta X$ be a homeomorphism such that $\Psi(p)=q$, then for the induced isomorphism $\Psi$ on $A^{\nu}$ we have $\Psi(M_{A^{\nu}}^{p})=M_{A^{\nu}}^{q}$
\end{lemma}

\begin{proof}
Theorem 3.3 of \cite{hs} shows that $M_{A^{\nu}}^{p}=\{f\in A^{\nu} ~|~ p\in S[\mathcal{Z}_{A^{\nu}}(f)]\}$ where $S[\mathcal{F}]=\bigcap\{\mbox{cl}_{\beta X} E~|~E\in \mathcal{F}\}$ and $\mathcal{F}$ is a $z$-filter. Mention should be made here that $\mathcal{Z}_{A^{\nu}}(f)$ is a $z$-filter in $X$ since $f$ is a non unit in $A^{\nu}$ [\cite{hs}, \S Lemma 1.4]. Again the fact $E\in\mathcal{Z}_{A^{\nu}}(f)$ shows that $\phi(E)\in \mathcal{Z}_{A^{\nu}}((f\circ\phi^{-1}))$ which evidently imply that if $p\in S[\mathcal{Z}_{A^{\nu}}(f)]$, i.e., $\bar{\phi}^{-1}(q)\in S[\mathcal{Z}_{A^{\nu}}(f)] $ then $q\in S[\mathcal{Z}_{A^{\nu}}(f\circ\phi^{-1})]$, i.e., $q\in S[\mathcal{Z}_{A^{\nu}}(\Psi(f))]$. Therefore it follows from Theorem 3.3, \cite{hs} that $\Psi(M_{A^{\nu}}^{p})=\{\Psi(f)~|~ p\in S[\mathcal{Z}_{A^{\nu}}(f)]\}$ $=\{\Psi(f)~|~q\in S[\mathcal{Z}_{A^{\nu}}(\Psi(f))]\}=M_{A^{\nu}}^{q}$.
\end{proof}

\begin{theorem}\label{characterization2} For a first countable tychonoff space $X$, let $\bar{\phi}$ be a homeomorphism from $\beta X$ onto itself such that $\bar{\phi}(p)=q$ for some $p,q\in \beta X-\upsilon_{A}X$, where $A(X) \in \Sigma(X)$ then for each $B_{p}^{\mathfrak{F}}(X)\in [A(X)]$ there exist $B_{q}^{\mathfrak{F}'}(X)\in [A(X)]$ such that $\bar{\phi}$ induces an isomorphism from $B_{p}^{\mathfrak{F}}(X)$ onto $B_{q}^{\mathfrak{F}'}(X)$ and from $\mathfrak{F}$ onto $\mathfrak{F}'$.
\end{theorem}
\begin{proof}
Recall the mapping $\Psi$ from the proof of Theorem \ref{induced iso} which also induce an isomorphism $\bar{\Psi}$ from $A^{\nu}/M_{A^{\nu}}^{p}$ to $A^{\nu}/M_{A^{\nu}}^{q}$. Let $\mathfrak{F}$ be the proper isomorphic copy of $A^{\nu}/M_{A^{\nu}}^{p}$ into itself. Then $\bar{\Psi}(\mathfrak{F})=\mathfrak{F}'$(say) is also a proper isomorphic copy of $A^{\nu}/M_{A^{\nu}}^{q}$ into itself. Let $\theta_{p}$ and $\theta_{q}$ are canonical maps from $A^{\nu}$ to $A^{\nu}/M_{A^{\nu}}^{p}$ and $A^{\nu}$ to $A^{\nu}/M_{A^{\nu}}^{q}$ respectively and let $\theta_{p}^{-1}(\mathfrak{F})=B_{p}^{\mathfrak{F}}(X)$ and $\theta_{q}^{-1}(\mathfrak{F})=B_{q}^{\mathfrak{F}'}(X)$. Then $B_{p}^{\mathfrak{F}}(X)$ and $B_{q}^{\mathfrak{F}'}(X)$ are intermediate rings and belongs to $[A(X)]$. It is quite clear now that the restriction of the isomorphism $\Psi$ to $B_{p}^{\mathfrak{F}}(X)$ gives an isomorphism from $B_{p}^{\mathfrak{F}}(X)$ onto $B_{q}^{\mathfrak{F}'}(X)$
\end{proof}

To ensure that nonisomorphic intermediate rings do exist in galore, we need the notion of type of a point in $\beta\mathbb{N}\setminus \mathbb{N}$ introduced by Frolik \cite{frlk} and recorded in \cite{RC}. We reproduce below a few relevant information about this notion from the monograph [\cite{RC}, \S 3.41,4.12]. Each permutation $\sigma:\mathbb{N}\longrightarrow\mathbb{N}$ extends to a homeomorphism $\sigma^*:\beta\mathbb{N}\longrightarrow\beta\mathbb{N}$, conversely if $\Phi:\beta\mathbb{N}\longrightarrow\beta\mathbb{N}$ is a homeomorphism then $\Phi|_{\mathbb{N}}$ is a permutation of $\mathbb{N}$, because $\Phi$ takes isolated points to isolated points and the points of $\mathbb{N}$ are the only isolated points of $\beta \mathbb{N}$. Therefore $\Phi=\sigma^*$ for unique permutation $\sigma=\Phi|_{\mathbb{N}}$ of $\mathbb{N}$.

\begin{definition}
For two points $p,q\in\beta\mathbb{N}\setminus\mathbb{N}$, we write $p\sim q$ when there exist a permutation $\sigma $ on $\mathbb{N}$ such that $\sigma^*(p)=q$. $\sim$ is an equivalence relation on $\beta\mathbb{N}\setminus \mathbb{N}$. Each equivalence class of elements of $\beta\mathbb{N}\setminus \mathbb{N}$ is called a type of ultrafilters on $\mathbb{N}$
\end{definition}
\begin{theorem}[Frolik,\cite{RC}]\label{fr}
There exists $2^c$ many types of ultrafilters on $\mathbb{N}$
\end{theorem}
\begin{theorem}[Frolik,\cite{RC}]\label{frr}
If $\mathbb{N}$ is $C$-embedded in $X$, then $\mbox{cl}_{\beta X}\mathbb{N}\setminus\mathbb{N}\subseteq \beta X\setminus X$, essentially $\beta\mathbb{N}\setminus\mathbb{N}\subseteq \beta X\setminus X$. If now $h:\beta X\longrightarrow \beta X$ is a homeomorphism onto $\beta X$, and $p,q\in \beta\mathbb{N}\setminus\mathbb{N}$ are such that $h(p)=q$, then $p$ and $q$ belongs to the same type of ultrafilters on $\mathbb{N}$.
\end{theorem}
 We now use these two theorem of Frolik to establish the last main result of the present paper.
\begin{theorem}\label{finish}
Let $X$ be a first countable noncompact realcompact space. Then there exist at least $2^c$ many intermediate subrings of $[A(X)]$, no two of which are isomorphic.
\end{theorem}
\begin{proof}
Since $X$ is a noncompact realcompact space it is not pseudocompact. Hence $X$ contains a copy of $\mathbb{N}$, $C$-embedded in $X$ [1.21, \cite{lm}]. As every $C$-embedded countable subset of a Tychonoff space is a closed subset of it [3,B3,\cite{lm}], it follows that $\mbox{cl}_{\beta X}\mathbb{N}\setminus \mathbb{N}\subseteq \beta X\setminus X$ essentially $\beta\mathbb{N}\setminus \mathbb{N}\subseteq \beta X\setminus X$. The result of Theorem \ref{fr} assures that there exist a subset $S$ of $\beta \mathbb{N}\setminus \mathbb{N}$, consisting exactly one member from each type with the property that $\lvert S\rvert =2^c$. Let $p$ and $q$ be two distinct points of the set $S$. Then it follows from Theorem \ref{frr} that no homeomorphism from $\beta X$ to $\beta X$ can exchange $p$ and $q$. We now use Theorem \ref{characterization} \& \ref{characterization2} to conclude that the rings $B_{p}^{\mathfrak{F}}(X)$ and $B_{q}^{\mathfrak{F}}(X)$ are not isomorphic. Hence the theorem follows.
\end{proof}

\end{document}